\documentclass[11pt]{article}

\usepackage[margin=1.3in]{geometry}
\usepackage[utf8]{inputenc}
\usepackage[UKenglish]{babel}

\usepackage{amsmath}
\usepackage{amsthm}
\usepackage{amsfonts}
\usepackage{amssymb}

\newtheorem{theorem}{Theorem}
\newtheorem{lemma}{Lemma}
\newtheorem{corollary}{Corollary}

\newtheorem{conjecture}{Conjecture}

\usepackage{todonotes}
\usepackage{enumitem}
\usepackage{multicol}
\usepackage{hyperref}
\usepackage{cleveref}
\usepackage{autonum}
\DeclareMathAlphabet{\mathpzc}{OT1}{pzc}{m}{it}

\allowdisplaybreaks

\usepackage{mathtools}
\DeclarePairedDelimiterX{\infdivx}[2]{(}{)}{#1 \,\delimsize\|\, #2}

\let\oldfootnote\footnote
\def\footnote{\ifhmode\unskip\fi\oldfootnote}

\DeclareMathOperator*{\E}{E}
\newcommand{\Ep}[1]{\E[#1]}

\usepackage[round]{natbib}

\title{Sharp and Simple Bounds for the Raw\\ Moments of the Binomial and Poisson Distributions}
\author{Thomas D. Ahle
\,\,\,
thomas@ahle.dk
\\
University of Copenhagen, BARC, Facebook
}

\begin{document}

\maketitle

\begin{abstract}
   We prove the inequality $\Ep{(X/\mu)^k}
   \le
   (\frac{k/\mu}{\log(1+k/\mu)})^k
   \le
   \exp(k^2/(2\mu))$ for sub-Poissonian random variables $X$, such as Binomially or Poisson distributed variables, with mean $\mu$.
   The asymptotic behaviour $\Ep{(X/\mu)^k}=1+O(k^2/\mu)$ 
   matches a lower bound of $1+\Omega(k^2/\mu)$
   for small $k^2/\mu$.
   This improves over previous uniform raw moment bounds by a factor exponential in $k$.
\end{abstract}

\section{Introduction}

Suppose we sample an urn of $n$ balls, each coloured \emph{red} with probability $p$ and otherwise \emph{blue}.
What is the probability that a sample of $k$ balls, with replacement, \emph{from this urn} consists of only red balls?
Such questions are of interest to sample-efficient statistics and the derandomisation of algorithms.


If $R\sim \text{Binomial}(n,p)$ denotes the number of red balls in the urn, the probability of drawing a single red ball from the urn is $R/n$.
Thus, the probability that a sample of $k$ balls from the urn is all red is given by $(R/n)^k$, or $P=\Ep{(R/n)^k}$ when the probability is taken over both sample phases.
Whenever the urn is large ($n$ is large), $R/n$ concentrates around $p$, so sampling from the urn is equivalent to sampling from the original distribution and $P\approx p^k$.
Indeed, from Jensen's inequality, we can see that $p^k$ is always a lower bound: $P = \Ep{(R/n)^k} \ge \Ep{(R/n)}^k = p^k$.
Previous authors have shown a nearly matching upper bound of $C^k p^k$ in the range
$k/(np) = O(1)$ for some constant $C>1$.
(See \cref{eq:latala} below for details.)
In this note, we improve the upper bound to
$P \le p^k (1 + k/(2np))^k$,
which shows that when $k = o(\sqrt{np})$,
the factor $C^k$ can be replaced by just $1+o(1)$.

\subsection{Related work}

One direct approach to computing the Binomial moments expands them using the Stirling numbers of the second kind:
$
\Ep{X^k} = \sum_{i=0}^k{k\brace i}n^{\underline{i}}p^i
$,
where $n^{\underline{i}}=n(n-1)\cdots(n-i+1)$.
This equality can be derived as a sum of the much easier to compute ``factorial moments'', $\Ep{X^{\underline{k}}} = n^{\underline{k}}p^k$.
See \cite{knoblauch2008closed} for details.
Taking the leading two terms of the sum, one finds that
$
\Ep{X^k} = (np)^k\left(1 + \binom{k}{2}\frac{1-p}{np} + O(1/n^2)\right)
   $ as $ n\to\infty$. 
However, this approach does not work when $k$ is not constant with respect to $n$.
Similarly, for the Poisson distribution, the moments can be expressed as the so-called Bell (or \citeauthor{touchard1939cycles}) polynomials in $\mu$:
$\Ep{X^k} = \sum_{i=0}^k {k\brace i} \mu^i $.
This sum gives a simple lower bound $\Ep{X^k} \ge {k\brace k}\mu^k + {k\brace k-1}\mu^{k-1} = \mu^k( 1+ \frac{k(k-1)}{2\mu})$, matching our upper bound asymptotically when $k=O(\sqrt{\mu})$.
However, as in the Binomial case, the sum does not easily yield a uniform bound.
We give the details of both lower bounds in \Cref{sec:lower}.

A different approach uses the powerful results on moments of independent random variables by \cite{latala1997estimation} and \cite{pinelis1995optimum}.
In the case of Binomial and Poisson random variables, they yield:
\[
   \left(c\,\frac{k/\mu}{\log(1+k/\mu)}\right)^k
   \le
   \Ep{(X/\mu)^k}
   \le
   \left(C\,\frac{k/\mu}{\log(1+k/\mu)}\right)^k
   \label{eq:latala}
\]
for some universal constants $c < 1 < C$.
The bound is tight up to the factor $(C/c)^k$, which is negligible when the overall growth is $O(k^k)$.
However, when $k/\mu\to 0$, we expect the upper bound to be 1, and so the factor $C^k$ in the upper bound can be overwhelmingly large.

A third option is to use a Rosenthal bound, such as the following by \cite{berend2010improved},
\citep[see also][]{johnson1985best}:
\[
   \Ep{X^k} \le B_k \max\{\mu, \mu^k\}.
   \label{eq:max}
\]
Here, $B_k$ is the $k$th Bell number, which Berend and Tassa show satisfies the uniform bound $B_k < \big(\frac{0.792 k}{\log(k+1)}\big)^k$.
For large $k$, a precise asymptotic bound, $B_k^{1/k} = \frac{k}{e\log k}(1+o(1))$, is given by \citep[e.g.][]{de1981asymptotic, ibragimov1998exact}.
Unfortunately, the Rosenthal bound is incomparable to the other bounds in this paper when $\mu < 1$, as it grows with $\mu$ rather than $\mu^k$.
However, for $\mu\ge 1$ and integral, we show a matching asymptotic lower bound in the second half of \Cref{sec:lower}.
That indicates that the upper bound of this paper could be improved by a factor $e^{-k}$ for large $k$.

Finally, \cite{ostrovsky2017non} give another asymptotically sharp bound in a recent preprint.
Using a technique based on moment generating functions, similar to this paper, they bound the Bell polynomial, which as discussed above, is equivalent to bounding the moments of a Poisson random variable.
The bound holds when $k\ge2\mu$:
\[
   \Ep{(X/\mu)^k}^{1/k} \le \frac{k/\mu}{e\log(k/\mu)}\left(1 + C(\mu)\frac{\log\log(k/\mu)}{\log(k/\mu)}\right)
   \quad\text{if }k\ge 2\mu,
   \label{eq:ostrovsky}
\]
where $C(\mu)>0$ is some ``constant'' depending only on $\mu$.
In the range $k<2\mu$, \citeauthor{ostrovsky2017non} only gives the bound $\Ep{(X/\mu)^k}\le 8.9758^k$, so similarly to the other bounds presented, it loses an exponential factor in $k$ compared to \Cref{thm} below, for smaller $k$.

\section{Bounds}

The theorem considers ``sub-Poissonian'' random variables, which are variables $X$, satisfying the requirement $\Ep{\exp(t X)} \le \exp(\mu(e^t-1))$.
Such sub-Poissonian include many simple distributions, such as the Poisson or Binomial distribution.
We give more examples in~\Cref{sec:sub}.

\begin{theorem}\label{thm}
   Let $X$ be a non-negative random variable with mean $\mu>0$ and moment-generating function $\Ep{\exp(t X)}$ bounded by $\exp(\mu(e^t-1))$ for all $t>0$.
   Then for all $k > 0$ and any $\alpha>0$:
   \[
      \Ep{(X/\mu)^k} \le \left(\frac{k/\mu}{e^{1-\alpha}\log(1+\alpha k/\mu)}\right)^k.
   \]
\end{theorem}

The theorem has a free parameter, $\alpha$, which is optimally set such that $1+\alpha k/\mu = e^{W(k/\mu)}$, where $W$ is the Lambert-W function, which is defined by $W(x)e^{W(x)}=x$.
\footnote{
   The Lambert-W function has multiple branches.
   We always refer to the main one (sometimes called the 0th), in which $W(x)$ and $x$ are both positive.
}
In practice the following two corollaries may be easier to work with.
\begin{corollary}\label{cor:up1}
   \[
      \Ep{(X/\mu)^k} \le
      \left(\frac{k/\mu}{\log(1+k/\mu)}\right)^k
      \le
      \left(1+\frac{k}{2\mu}\right)^k
      \le \exp\!\left(\frac{k^2}{2\mu}\right)
   .\]
\end{corollary}
\begin{proof}
   For the first inequality, set $\alpha=1$ in \Cref{thm}.
   The second bound, we use a standard logarithmic inequality, $\frac{x}{\log(1+x)}\le 1+x/2$ \citep[see e.g.][eq. 6]{topsoe2007some}.
   The last bound is the standard $1+x\le \exp(x)$.
\end{proof}
In the range $k=O(\sqrt{\mu})$
we show a matching lower bound of $1+\Omega(k^2/\mu)$ in \Cref{sec:lower}, \cref{eq:lower1}.
\begin{corollary}\label{cor:up2}
   Let $x=k/\mu$, then
   \[
      \Ep{(X/\mu)^k}^{1/k}
      \le
      \frac{x\, e^{1/\log(e+x)}}{e\log (1+x/\log(e+x))}
      =
      \frac{x}{e\log x}\left(1+O\!\left(\frac{\log\log x}{\log x}\right)\right)
      \quad
      \text{as $x\to\infty$.}
      \label{eq:loglog-lower}
   \]
\end{corollary}
\begin{proof}
   Take $\alpha=1/\log(e+x)$.
   For $x>0$ we have $\log(e+x)>0$ and so $\alpha>0$ as required by \Cref{thm}.
\end{proof}
\Cref{cor:up2} matches our lower bound in \cref{eq:lower-bound}, as well as \citeauthor{ostrovsky2017non} in \cref{eq:ostrovsky}, but without the restriction on the range of $k/\mu$.

\subsection{The proof}

Technically our bound is shown using the moment-generating function and some new sharp inequalities involving the Lambert-W function.
We will use the following lemma:
\begin{lemma}[\citealp*{hoorfar2008inequalities}]\label{lem:hoorfar}
   For all $y>1/e$ and $x > -1/e$,
   \[
      e^{W(x)} \le \frac{x+y}{1+\log y}.
      \label{eq:hoorfar}
   \]
\end{lemma}
We present an elementary proof of this fact for completeness:
\begin{proof}
   Starting from $1+t\le e^t$, substitute $\log(y)-t$ for $t$ to get
   $1+\log y - t \le y e^{-t}$.
   Multiplying by $e^t$ we get $e^t(1+\log y) \le t e^t + y$.
   Let $t=W(x)$ s.t. $t e^t=x$.
   Rearranging, we get~\cref{eq:hoorfar}.
\end{proof}
Taking $y=e^{W(x)}$ in \cref{eq:hoorfar} makes the two sides equal,
so we can think of \Cref{lem:hoorfar} as a way to turn a rough estimate into an upper bound.
\vspace{.5em}

We apply \Cref{lem:hoorfar} to show a new bound on $W(x)$ in a similar style.
This lemma will be the main ingredient in proving \Cref{thm}.
\begin{lemma}\label{lem:W-bound}
   For all $y>1$ and $x > 0$,
   \[
      \frac1{W(x)}+W(x)
      \le
    \frac{y}{x} +
    \log\!\left(\frac{x}{\log y}\right),
   \]
   with equality if $y=e^{W(x)}$.
\end{lemma}
\begin{proof}
   The proof uses the identities $W(x)=\log(\tfrac{x}{W(x)})$ and $\frac{1}{W(x)}=\frac1x\exp(W(x))$ which are simple rewritings of the definition $W(x)e^{W(x)}=x$.
   The main idea is to introduce a new variable $z>0$, to be determined later, which allows us to control the effect of applying the logarithmic inequality $\log x \ge 1-1/x$.
   We also use \Cref{lem:hoorfar}, which introduces another new variable $y>1$ to be determined.

   We bound:
   \begin{align}
      \frac1{W(x)}+W(x)
      &=
      \frac1{W(x)}+\log\!\left(\frac{x}{W(x)}\right)
    \\&=
      \frac1{W(x)}+\log\!\left(\frac{x}{z}\right)-\log\!\left(\frac{W(x)}{z}\right)
    \\&\le
    \frac1{W(x)}+\log\!\left(\frac{x}{z}\right)-\left(1-\frac{z}{W(x)}\right)
    \\&=
    \frac{1+z}{W(x)}-1+\log\!\left(\frac{x}{z}\right)
    \\&=
    e^{W(x)}\frac{1+z}{x}-1+\log\!\left(\frac{x}{z}\right)
    \\&\le
    \frac{x+y}{1+\log(y)}\frac{1+z}{x}
    -1+\log\!\left(\frac{x}{z}\right).
    \\&=
    \frac{y}{x} +
    \log\!\left(\frac{x}{\log y}\right).
   \end{align}

   Here the last two steps come from
   the inequality \cref{eq:hoorfar} in its general form,
   and the substitution $z=\log y$.
   We can check that equality follows all the way through if we let $y=e^{W(x)}$.
\end{proof}

We are now ready to prove the main theorem of the paper:
\begin{proof}[Proof of \Cref{thm}.]
   Let $m(t) = \Ep{\exp(t X)}$ be the moment-generating function.
   We will bound the moments of $X$ by
   \[
      \Ep{X^k} \le m(t)\left(\frac{k}{et}\right)^k,
      \label{eq:bound1}
   \]
   which holds for all $k\ge 0$ and $t>0$.
   This follows from the basic inequality $1+z\le e^z$, where we substitute $tz/k-1$ for $z$ to get $tz/k \le e^{tz/k-1} \implies z^k \le e^{tz}(k/(et))^k$.
   Letting $z=X$ and taking expectations, we get \cref{eq:bound1}.

   We now define $x=k/\mu$ and take $t$ such that $t e^t=x$.
   In the notation of the Lambert-W function, this means $t=W(x)$.
   We note that $t>0$ whenever $x>0$.
   We proceed to bound the moments of $X/\mu$ using \cref{eq:bound1}:
   \begin{align}
      \Ep{(X/\mu)^k}
      &\le
      m(t)\left(\frac{k}{et}\right)^k\mu^{-k}
      \\&\le
      \exp\!\big(\mu(e^t-1)\big)\left(\frac{k}{e \mu t}\right)^k
      \\&=
      \exp\!\big(\mu(x/t-1)\big)\left(\frac{e^t}{e}\right)^k
      \label{eq:subs}
      \\&=
      \exp\!\big((k/x)(x/t-1)+k(t-1)\big)
      \\&= \exp(k f(x))
      \label{eq:exp(kfB)}
      ,
   \end{align}
   where we define $f(x) \coloneqq 1/t-1/x+t-1$.
   Here \cref{eq:subs} came from the simple rewriting of the definition of $t$, $1/t=e^t/x$

   We continue to bound $f(x)$ using \Cref{lem:W-bound}:
   \begin{align}
      f(x)
      &=
      \frac1{W(x)} + W(x) - 1 - \frac1x
      \\&\le
      \frac{y}{x} + \log\left(\frac{x}{\log y}\right)
       - 1 - \frac1x
      \\&=
      \alpha-1 + \log\!\left(\frac{x}{\log(1+\alpha x)}\right),
      \label{eq:bound2}
   \end{align}
   taking $y=1+\alpha x$, which is greater than 1 when $\alpha$ and $x$ are both greather than 0.

   Backing up, we have shown
   \[
      \Ep{(X/\mu)^k} \le \exp(kf(x)) \le \left(\frac{x}{e^{1-\alpha}\log(1+\alpha x)}\right)^k,
   \]
   which finishes the proof.
\end{proof}

\subsection{Lower bound}\label{sec:lower}

As mentioned in the introduction, the expansion for the Poisson moments
$\Ep{X^k} = \sum_{i=0}^k {k\brace i} \mu^i$
gives a simple lower bound by taking the two highest terms.
We note that ${k\brace k}=1$ and ${k\brace k-1}=\binom{k}{2}$
to get
\[
   \Ep{X^k} \ge \mu^k \left(1 + \frac{k(k-1)}{2\mu}\right),
   \label{eq:lower1}
\] matching \Cref{thm} asymptotically for $k=O(\sqrt{\mu})$.

The expansion for Binomial moments
$\Ep{X^k} = \sum_{i=0}^k {k\brace i} n^{\underline{i}} p^i$
yields a similar lower bound
\begin{align}
   \Ep{X^k}
   &\ge
   n^{\underline{k}}p^k
   + \binom{k}{2}n^{\underline{k-1}}p^{k-1}
 \\&=
 (np)^k
   \left(\frac{n^{\underline{k}}}{n^k}\right)
   \left(1 + \binom{k}{2}\frac{1}{(n-k+1)p}\right)
 \\&=
 (np)^k
   \left(\prod_{i=0}^{k-1}1-\frac{i}{n}\right)
   \left(1 + \binom{k}{2}\frac{1}{(n-k+1)p}\right)
 \\&\ge
 (np)^k
   \left(1-\binom{k}{2}\frac1n\right)
   \left(1 + \binom{k}{2}\frac{1}{np}\right)
 \\&=
 (np)^k
   \left(1 + \binom{k}{2}\frac{1-p}{np}
      \left(1 -\binom{k}{2}\frac{1}{n}\right)
   \right),
\end{align}
which matches \Cref{thm} for $k=O(\sqrt{\mu})$ and $p$ not too close to 1.

\vspace{.5em}

We will investigate some more precise lower bounds as $k/\mu$ gets large.
As mentioned briefly in the introduction, there is a correspondence between the moments of a Poisson random variable and the Bell polynomials defined by $B(k,\mu) = \sum_i{k\brace i}\mu^i$.
In particular, $\Ep{X^k} = B(k,\mu)$, if $\mu$ is the mean of the Poissonian random variable.
The Bell polynomials are so named because $B(k,1)$ is the $k$th Bell number.
By Dobiński's formula $B(k,1) = \frac{1}{e}\sum_{i=0}^\infty \frac{i^k}{i!}$ the Bell numbers are generalised for real $k$.
We write these as $B_x=B(x,1)$.

We give a lower bound for $\Ep{(X/\mu)^k}$ by showing the following simple connection between the Bell polynomials and Bell numbers:

\begin{theorem}\label{thm:lower}
   Let $k$ be a positive real number and $\mu\ge 1$ be an integer.
   Then
   \[
      B(k,\mu)/\mu^k
      \ge
      B_{k/\mu}^{\mu}.
   \]
\end{theorem}

While the proof below assumes $\mu$ is an integer, we will conjecture~\Cref{thm:lower} to be true for any $\mu\ge1$.
Now by \citeauthor{de1981asymptotic}'s \citeyearpar{de1981asymptotic} asymptotic expression for the Bell numbers:
\[
   \Ep{(X/\mu)^k}^{1/k}\ge
   B_{k/\mu}^{\mu/k}
   = \frac{k/\mu}{e\log(k/\mu)}\left(1+\Theta\left(\frac{\log\log(k/\mu)}{\log(k/\mu)}\right)\right)
   \quad\text{as } k/\mu\to\infty.
   \label{eq:lower-bound}
\]
matching our upper bound, \cref{eq:loglog-lower}, the upper bound of Ostrovsky and Sirota, \cref{eq:ostrovsky}, for large $k$, as well as Latała's uniform lower bound with a different constant.

\begin{proof}[Proof of \Cref{thm:lower}]
   Let $X, X_1, \dots, X_\mu$ be i.i.d. Poisson variables with mean $1$,
   then $S=\sum_{i=1}^\mu X_i$ is Poisson with mean $\mu$.
   We write $\|X\|_k = \Ep{X^k}^{1/k}$.
   Then by the AG inequality:
   \[
      \|S/\mu\|_k
      =
      \left\|\frac{1}{\mu}\sum_{i=1}^\mu X_i\right\|_k
      \ge
      \left\|\bigg(\prod_{i=1}^\mu X_i\bigg)^{1/\mu}\right\|_k
      =
      \left\|\prod_{i=1}^\mu X_i\right\|^{1/\mu}_{k/\mu}
      =
      \left(\prod_{i=1}^\mu\|X_i\|_{k/\mu}\right)^{1/\mu}
      =
      \|X\|_{k/\mu}.
      \label{eq:ag}
   \]
   Since $X$ has mean 1 we have $\|X\|_{k/\mu} = B_{k/\mu}^{\mu/k}$,
   and as $S$ has mean $\mu$ we have $\|S/\mu\|_k = B(k,\mu)^{1/k}/\mu$.
   Thus, taking $k$th powers, \cref{eq:ag} is what we wanted to show.
\end{proof}

For small $k/\mu$ this bound is less interesting since $B_{x}\to 0$ as $x\to 0$, rather than 1 as our upper bound.
However, it is pretty tight, as we conjecture by the following matching upper bound in terms of the Bell numbers:
\begin{conjecture}
   For all $k>0$ and $\mu\ge 1$,
\[
   B_{k/\mu}^{1/(k/\mu)}
   \le
   \frac{B(k,\mu)^{1/k}}{\mu}
   \le
   B_{k/\mu+1}^{1/(k/\mu+1)}
   .
\]
Furthermore, for $0<\mu\le 1$, $
   \frac{B(k,\mu)^{1/k}}{\mu}
   \le
   B_{k/\mu}^{1/(k/\mu)}
   $.
\end{conjecture}
While the upper bound appears true numerically, it can't follow from our moment-generating function bound \cref{eq:exp(kfB)}, since it drops below that for $k/\mu$ bigger than 40.
The conjectured upper bound is even incomparable with our \Cref{thm}, since it is slightly above $\frac{k/\mu}{\log(1+k/\mu)}$ for very small $k/\mu$.
The conjectured bound is weaker than \cref{eq:max} by \cite{berend2010improved} in the region $k<2$ and $\mu<1$, but for all other parameters, it is substantially tighter.

\section{Sub-Poissonian Random Variables}\label{sec:sub}
We call a non-negative random variable $X$ sub-Poissonian if $\Ep{X}=\mu$ and the moment-generating function, mgf., $\Ep{\exp(t X)} \le \exp(\mu(e^t-1))$ for all $t>0$.
We will briefly show that this notion includes all sums of bounded random variables, such as the Binomial distribution.

If $X_1, \dots, X_n$ are sub-Poissonian with mgf. $m_1(t), \dots, m_n(t)$ and mean $\mu_1, \dots, \mu_n$ respectively, then $\sum_i X_i$ is sub-Poissonian as well, since
\[
   \E\!\big[\exp\!\big(t\sum_i X_i\big)\big] =
   \prod_i m_i(t)
   \le \prod_i \exp\!\left(\mu_i(e^t-1)\right)
   = \exp\!\Big(\big(\sum_i \mu_i\big)\big(e^t-1\big)\Big).
\]
Next, a random variable bounded in $[0,1]$ with mean $\mu$ has mgf.
\[
   \Ep{\exp(t X)}
   = 1 + \sum_{k=1}^\infty \frac{t^k \Ep{X^k}}{k!}
   \le 1 + \mu \sum_{k=1}^\infty \frac{t^k \Ep{1^{k-1}}}{k!}
   = 1 + \mu(e^t-1)
   \le \exp(\mu(e^t-1)).
\]
Hence if $X=X_1+\dots+X_n$ where each $X_i\in[0,1]$ we have $\mu=\Ep{X}=\sum_i \Ep{X_i}$ and by \Cref{thm}
that $\Ep{(X/\mu)^k} \le \frac{k/\mu}{\log(k/\mu+1)}$.
In particular this captures sum of Bernoulli variables with distinct probabilities.

An example of a non-sub-Poissonian distribution is the geometric distribution with mean $\mu$.
This has moment generating function $m(t)=\frac{1}{1-\mu(e^t-1)}$, which is larger than $\exp(\mu(e^t-1))$ for all $t>0$.
However, likely, similar methods to those in the proof of \Cref{thm} will still apply to bound its moments.

\section{Acknowledgements}
The author would like to thank Robert E. Gaunt for his encouragement and helpful suggestions.

\bibliographystyle{plainnat}
\bibliography{bi-moments}

\begin{thebibliography}{11}
\providecommand{\natexlab}[1]{#1}
\providecommand{\url}[1]{\texttt{#1}}
\expandafter\ifx\csname urlstyle\endcsname\relax
  \providecommand{\doi}[1]{doi: #1}\else
  \providecommand{\doi}{doi: \begingroup \urlstyle{rm}\Url}\fi

\bibitem[Berend and Tassa(2010)]{berend2010improved}
Daniel Berend and Tamir Tassa.
\newblock {Improved bounds on Bell numbers and on moments of sums of random
  variables}.
\newblock \emph{Probability and Mathematical Statistics}, 30\penalty0
  (2):\penalty0 185--205, 2010.

\bibitem[{d}e Bruijn(1981)]{de1981asymptotic}
Nicolaas~Govert {d}e Bruijn.
\newblock \emph{{Asymptotic methods in analysis}}, volume~4.
\newblock Courier Corporation, 1981.

\bibitem[Hoorfar and Hassani(2008)]{hoorfar2008inequalities}
Abdolhossein Hoorfar and Mehdi Hassani.
\newblock {Inequalities on the Lambert W function and hyperpower function}.
\newblock \emph{J. Inequal. Pure and Appl. Math}, 9\penalty0 (2):\penalty0
  5--9, 2008.

\bibitem[Ibragimov and Sharakhmetov(1998)]{ibragimov1998exact}
Rustam Ibragimov and Sh~Sharakhmetov.
\newblock {On an Exact Constant for the Rosenthal Inequality}.
\newblock \emph{Theory of Probability \& Its Applications}, 42\penalty0
  (2):\penalty0 294--302, 1998.

\bibitem[Johnson et~al.(1985)Johnson, Schechtman, and Zinn]{johnson1985best}
William~B Johnson, Gideon Schechtman, and Joel Zinn.
\newblock {Best Constants in Moment Inequalities for Linear Combinations of
  Independent and Exchangeable Random Variables}.
\newblock \emph{The Annals of Probability}, 13\penalty0 (1):\penalty0 234 --
  253, 1985.

\bibitem[Knoblauch(2008)]{knoblauch2008closed}
Andreas Knoblauch.
\newblock {Closed-Form Expressions for the Moments of the Binomial Probability
  Distribution}.
\newblock \emph{SIAM Journal on Applied Mathematics}, 69\penalty0 (1):\penalty0
  197--204, 2008.

\bibitem[Lata{\l}a(1997)]{latala1997estimation}
Rafa{\l} Lata{\l}a.
\newblock {Estimation of moments of sums of independent real random variables}.
\newblock \emph{The Annals of Probability}, 25\penalty0 (3):\penalty0
  1502--1513, 1997.

\bibitem[Ostrovsky and Sirota(2017)]{ostrovsky2017non}
Eugene Ostrovsky and Leonid Sirota.
\newblock {Non-asymptotic estimation for Bell function, with probabilistic
  applications}.
\newblock \emph{arXiv preprint arXiv:1712.08804}, 2017.

\bibitem[Pinelis(1995)]{pinelis1995optimum}
Iosif Pinelis.
\newblock {Optimum bounds on moments of sums of independent random vectors}.
\newblock \emph{Siberian Adv. Math}, 5\penalty0 (3):\penalty0 141--150, 1995.

\bibitem[Tops{\o}e(2007)]{topsoe2007some}
Flemming Tops{\o}e.
\newblock {Some bounds for the logarithmic function}.
\newblock \emph{Inequality theory and applications}, 4\penalty0 (01), 2007.

\bibitem[Touchard(1939)]{touchard1939cycles}
Jacques Touchard.
\newblock {Sur les cycles des substitutions}.
\newblock \emph{Acta Mathematica}, 70\penalty0 (1):\penalty0 243--297, 1939.

\end{thebibliography}

\end{document}